\documentclass[11pt,reqno]{amsproc}
\usepackage[top=0.8in, bottom=0.8in, left=0.8in, right=0.8in]{geometry}
\usepackage{hyperref}
\usepackage{mathrsfs}
\usepackage{amsmath, amsfonts, mathtools, amsthm}
\usepackage{color}
\mathtoolsset{showonlyrefs}

\newtheorem{theorem}{Theorem}[section]
\newtheorem{lemma}[theorem]{Lemma}
\newtheorem{proposition}[theorem]{Proposition}

\newtheorem{definition}[theorem]{Definition}

\newtheorem{remark}[theorem]{Remark}

\newcommand{\R}{\mathbb{R}}

\newcommand{\Z}{\mathbb{Z}}
\newcommand{\N}{\mathbb{N}}
\newcommand{\T}{\mathbb{T}}
\newcommand{\de}{\mbox{ d}}
\newcommand{\p}{\partial}
\newcommand{\boldb}{\mathbf{B}}
\newcommand{\dive}{\mathrm{div} \,}
\newcommand{\LL}{\mathbb{L}}

\begin{document}

\title{Magnetic Relaxation of a Voigt--MHD System}

\author{Peter Constantin and Federico Pasqualotto}

\begin{abstract}
    We construct solutions of the magnetohydrostatic (MHS) equations in bounded domains and on the torus in three spatial dimensions, as infinite time limits of Voigt approximations of viscous, non-resistive incompressible magnetohydrodynamics equations. The Voigt approximations modify the time evolution without introducing artificial viscosity. We show that the obtained MHS solutions are regular, nontrivial, and are not Beltrami fields.
\end{abstract}

\maketitle

\section{Introduction}
We consider incompressible {magnetohydrodynamics} equations (MHD)~\cite{alfven} describing the coupled evolution of a velocity vector field $u(x,t)$ and of a magnetic field $B(x,t)$. A vector field $B(x)$ is a solution of ideal {magnetohydrostatic} equations (MHS), if it is a time independent solution of the ideal MHD system with vanishing velocity, $u\equiv0$. Such a vector field is called an {ideal MHS equilibrium}. 

The construction of ideal MHS equilibria is of great importance in connection to the design of {nuclear fusion} devices such as tokamaks and stellarators. The classical variational approach of obtaining MHS equilibria ~\cite{kruskalkulsrud} was  extended and amplified ~\cite{kraushudson}  to include favorable features of the solutions. MHS solutions obtained by the variational approach are not smooth. Solutions with nontrivial but discontinuous locally constant pressure were constructed in ~\cite{brunolaurence} using a KAM argument which pieces together Beltrami fields across current sheets. This type of solution can be obtained in complex geometrical configurations ~\cite{enciso2021}.  Continuous stepped pressure profiles have been obtained numerically ~\cite{kraushudson} by  variational methods.

Explicit axisymmetric solutions to ideal MHS equations with discontinuous pressure are classical ~\cite{hill} and  smooth compactly supported axisymmetric ideal MHS equilibria have been obtained more recently by hodograph and ODE methods ~\cite{gavrilov} and Grad-Shafranov equations ~\cite{constlavicol}. These solutions lack the physically desired property of a nonnegative pressure. Quasi-symmetric MHS equilibria which are smooth, not axisymmetric and have nontrivial nonnegative plasma pressure are relevant for the design of stellarators. Such solutions have been constructed by deforming smooth axisymmetric Grad-Shafranov solutions  ~\cite{cdg2020}, ~\cite{cdg2021}, but they require an additional small smooth forcing to be sustained.

Different from the variational  approach, the magnetic relaxation approach ~\cite{perspectives}, ~\cite{moffatt1, moffatt2}, 
seeks to obtain MHS solutions as the long time limits of evolutions of well chosen regularized systems. 
Smooth evolutions that preserve topology were devised, in the hope to obtain smooth MHS solutions with prescribed topological properties. 

Weak solutions to MHS have been constructed in~\cite{brenier1} by a long time limit.  The recent analysis of Moffatt's magnetic relaxation equations~\cite{beekie2021} shows that these equations are globally well posed in higher regularity spaces and may lead to regular MHS solutions in some cases.

It is known that smooth time dependent solutions of non-resistive MHD equations conserve the topology of magnetic lines.
Resistive MHD evolution or singularities might be responsible for topology change, leading to magnetic reconnection events ~\cite{taylor1986} which are real physical phenomena, occurring for instance in solar flares. 

Although magnetic reconnection is not well understood mathematically, the fact that the magnetic field $B$ in ideal MHD evolution might grow rapidly is known. Rigorous examples of infinite time growth are classical ~\cite{Yudovich74}. The subject has been widely investigated, see for instance~\cite{biskamp}. 

In this work we construct MHS equilibria as long time limits of certain Voigt regularizations of the MHD equations. Our work provides an algorithmic construction of MHS equilibria (or steady solutions of incompressible 3D Euler equations) as a long time limit starting from arbitrary initial data.

Voigt regularizations have been widely studied in the context of incompressible fluid dynamics~\cite{lariostiti},~\cite{caoetal2004}.  In the viscous case, these models are known in viscoelasticity as Kelvin--Voigt fluids, ~\cite{oskolkov73}. Due to their favorable regularity properties, the Voigt-regularized equations have been successfully used to construct statistical solutions ~\cite{ramostiti} and ~\cite{levantramostiti}, where suitable convergence results to statistical solutions were proved as the regularization parameter tends to zero.
Regularity properties of Voigt regularized models were obtained  in ~\cite{lariostiti}), and in magnetohydrodynamic contexts  in ~\cite{linshitztiti2007} and~\cite{lariostiti2014}).

The incompressible, viscous, resistive MHD equations in three space dimensions with periodic boundary conditions for the velocity $u$ and magnetic field $B$,  $u(x,t), B(x,t):  \T^3 \times [0,T] \to \R^3$, are 
\noeqref{eq:MHDintro2} \noeqref{eq:MHDintro3}
\begin{align}
&\p_t u + u \cdot \nabla u + \nabla p = B \cdot \nabla B  + \nu \Delta u,\label{eq:MHDintro1} \tag{MHD-1}\\
&\p_t B + u \cdot \nabla B - B \cdot \nabla u = \mu \Delta B,\label{eq:MHDintro2}\tag{MHD-2}\\
&\dive u =0, \qquad \dive B = 0,\label{eq:MHDintro3}\tag{MHD-3}\\
&(u,B)|_{t = 0} = (u_0, B_0).\label{eq:MHDintro4} \tag{MHD-4}
\end{align}
Here, $\nu \geq 0$ is the kinematic viscosity of the fluid, $\mu \geq 0$ is the magnetic resistivity, and $p = p(x,t)$ is the hydrodynamical pressure.

A time independent vector field $B: \T^3 \to \R^3$ is an ideal magnetohydrostatic (MHS)  equilibrium if it satisfies the system~\eqref{eq:MHDintro1}--\eqref{eq:MHDintro4} with $u \equiv 0$ in the case $\mu = \lambda = 0$, that is, if it satisfies the equations
\begin{align}\label{eq:steady3deuler1}\tag{MHS-1}
&B \cdot \nabla B - \nabla p = 0,\\
&\dive B = 0.\label{eq:steady3deuler2}\tag{MHS-2}
\end{align}
with a sufficiently regular pressure function $p$.

The Voigt-regularized MHD system we consider is
\noeqref{eq:voigtMHDintro2}\noeqref{eq:voigtMHDintro3}
\begin{align}
&\p_t \LL u + u \cdot \nabla u + \nabla q = B \cdot \nabla B  + \nu \Delta u,\label{eq:voigtMHDintro1}\tag{VMHD-1}\\
&\p_t \LL B + u \cdot \nabla B = B \cdot \nabla u,\label{eq:voigtMHDintro2}\tag{VMHD-2}\\
&\dive u = 0, \qquad \dive B = 0,\label{eq:voigtMHDintro3}\tag{VMHD-3}\\
&(u, B)|_{t = 0} = (u_0, B_0).\label{eq:voigtMHDintro4}\tag{VMHD-4}
\end{align}

The kinematic viscosity is positive, $\nu > 0$. The magnetic resistivity is set to zero $\mu =0$.
 The regularization is applied to both the velocity field $u$ and the magnetic field $B$.  In the periodic case
 $\LL =  (-\Delta)^\alpha$, with $\alpha > 0$.

The formal a-priori energy inequality for system~\eqref{eq:voigtMHDintro1}--\eqref{eq:voigtMHDintro4} is, for all $t \geq 0$:
\begin{equation}\label{eq:enintro} \tag{EN}
\begin{aligned}
\|u(\cdot, t)\|^2_{{{\dot H}^\alpha}} + \|B(\cdot, t)\|^2_{{{\dot H}^\alpha}} + 2 \nu \int_0^t \| \nabla u(\cdot, s)\|^2_{L^2} \, \de s \leq C,
\end{aligned}
\end{equation}
where $C > 0$ is a constant which depends on initial data, and ${{\dot H}^\alpha}$ is the usual homogeneous Sobolev space on $\T^3$, whose definition is recalled in Appendix~\ref{app:setupper}. From~\eqref{eq:enintro}, we note that the regularization on $B$ is inviscid (it does not dissipate energy for $B$), however the presence of the friction term in the momentum equation ensures that the energy of $u$ is dissipated.

There are two main reasons  we use this regularization. First, if $\alpha$ is sufficiently high, the resulting system has global solutions for large initial data, due to the strong control given by the energy inequality~\eqref{eq:enintro}.

Second, the regularization on the induction equation (the equation for $B$) gives additional compactness, which allows us to  pass to the limit in the expression $B \cdot \nabla B$  following from inequality~\eqref{eq:enintro}, due to the uniform boundedness of a high norm of $B$.  The price we have to pay for this very good compactness property is poor control on the topology of the limiting magnetic field.  In particular, this means that the induction equation no longer holds exactly, and therefore, along the evolution, the magnetic field might change its topology.

Our main results are as follows. A preliminary version of these Theorems was contained in the PhD thesis~\cite{pasqualotto2020}.

The first statement concerns the construction of ideal 3D MHS equilibria with periodic boundary conditions (solutions to equation~\eqref{eq:steady3deuler1}--\eqref{eq:steady3deuler2}) as long-time limits of the Voigt regularized system~\eqref{eq:voigtMHDintro1}--\eqref{eq:voigtMHDintro4}.

\begin{theorem}[Construction of MHS equilibria, periodic case] \label{thm:relax1} 
Consider $u_0, B_0 \in D((-\Delta)^{\alpha/2})$ divergence-free and mean-free vector fields, with $\alpha \geq 1$. Let $(u, B)$ the global solution to the initial value problem~\eqref{eq:voigtMHDintro1}--\eqref{eq:voigtMHDintro4} given by Theorem~\ref{thm:global}. Then, there is a sequence $\{t_n\}_{n\in \N}$, $t_n \to \infty$ as $n \to \infty$, and $B_\infty \in D((-\Delta)^{\alpha/2})$, such that 
\begin{enumerate}
    \item  $B(x,t_n)$ tends weakly in $D((-\Delta)^{\alpha/2})$ (strongly in any $D((-\Delta)^{\gamma/2})$ with $\gamma < \alpha$) to $B_\infty(x)$, and
    \item $B_\infty$ is a strong (in $L^2$) solution of the stationary Euler system:
\begin{equation}\label{eq:euler3dthm}
\begin{aligned}
&B_\infty \cdot \nabla B_\infty   - \nabla q_\infty = 0,\\
&\dive B_\infty = 0
\end{aligned}
\end{equation}
for a pressure function $q_\infty \in \dot H^1$. 
\end{enumerate}

In addition, the following modified magnetic helicity,
\begin{equation}
H[B](t):=\int_{\T^3}  B \cdot \LL \Psi \de x
\end{equation}
is constant as a function of time if $u, B$ solve~\eqref{eq:voigtMHDintro1}--\eqref{eq:voigtMHDintro4}. Here, $\Psi$ is a magnetic potential as defined in Lemma \ref{lem:magneticpot}, i.~e.~a vector field in $\Psi \in {L^\infty}(0,\infty; D((-\Delta)^{(\alpha+1)/2}))$ satisfying $\text{div }\Psi = 0$, $\nabla \times \Psi = B$. Therefore, if the initial data $(u_0, B_0)$ are such that $H[B_0](0) \neq 0$, the resulting $B_\infty$ is nontrivial.
\end{theorem}

In the second statement, we are constructing solutions to the MHS equations in bounded domains. We let $\Omega \subset \R^3$ be a bounded domain with smooth boundary $\p \Omega$, and we let $\gamma_0: H^\gamma(\Omega) \to H^{\gamma - \frac 12}(\p \Omega)$ be the trace map, where $\gamma > 0$. We recall that $A:= \mathbb{P}(-\Delta)$ is the Stokes operator on $\Omega$ (where $\mathbb{P}$ is the Leray projector) and, for $\gamma > 0$, we recall the spaces $D(A^{\gamma/2})$ (for the precise setup, we refer to Section~\ref{app:setupbdd}).

We consider the initial-boundary value problem:
\noeqref{eq:MHDVbddintro2}
\begin{align}
&\p_t u + A^{-\alpha}(u \cdot \nabla u - B \cdot \nabla B ) + \nu A^{-\alpha +1} u = 0,  \label{eq:MHDVbddintro1}\tag{VMHD-$\Omega$-1}\\
&\p_t B + A^{-\alpha}(u \cdot \nabla B - B \cdot \nabla u) = 0,\label{eq:MHDVbddintro2}\tag{VMHD-$\Omega$-2}\\
&(u, B)|_{t = 0} = (u_0, B_0). \label{eq:MHDVbddintro3}\tag{VMHD-$\Omega$-3}
\end{align}
Here, $\nu >0$ is a fixed parameter.

\begin{theorem}[Constructions of MHS equilibria, bounded domain case]\label{thm:bdd} Let $\alpha \geq 1$. Let $u_0, B_0 \in D(A^{\alpha/2})$. Under these assumptions, the problem~\eqref{eq:MHDVbddintro1}--\eqref{eq:MHDVbddintro3} admits a global-in-time strong solution $(u, B) \in L^\infty(0,\infty; D(A^{\alpha/2})) \times L^\infty(0,\infty; D(A^{\alpha/2}))$. Moreover, there is a sequence $\{t_n\}_{n\in \N}$, $t_n \to \infty$ as $n \to \infty$, and $B_\infty \in D(A^{\alpha/2})$, such that 
\begin{enumerate}
\item $B(x, t_n)$ tends weakly in $D(A^{\alpha/2})$ (and strongly in any $D(A^{\gamma/2})$ with $\gamma < \alpha$) to $B_\infty(x)$, and
\item $B_\infty$ is a strong solution of the stationary Euler system:
\noeqref{eq:euler3dthmbdd}
\begin{align}\label{eq:euler3dthmbdd}
&\mathbb{P}(B_\infty \cdot \nabla B_\infty )  = 0,\\
&B_\infty \in D(A^{\alpha/2}). \label{eq:bcsteady}
\end{align}
\item $B_{\infty}$ is not a force-free field. If $J_{\infty}\times B_{\infty} = 0$ with $J_{\infty}
=\nabla\times B_{\infty}$ then $B_{\infty} =0$. In particular, $B_{\infty}$ is not a Beltrami field.
\end{enumerate}
If $\alpha = 1$, and if $B_0$ admits a magnetic potential $\Psi_0$ such that
$$
H[B_0, \Psi_0] := \int_\Omega \LL B_0 \cdot \Psi_0 dx \neq 0,
$$
$B_\infty$ is nontrivial. 
\end{theorem}

\begin{remark}
The inclusion~\eqref{eq:bcsteady} encodes boundary conditions.
\end{remark}

\begin{remark}
Replacing, in the Voigt regularization term (both in $u$ and in $B$) the operator $A^\alpha$ by the operator $\tilde A := \exp(A)$, it is possible to construct smooth solutions to MHS with $B_\infty \in D(\exp(A)) \subset \cap_{\alpha > 0} D(A^{\alpha/2})$.
\end{remark}

We finally provide an example of gradient growth in MHD \emph{without imposing any regularization}. The example, in the context of the 3D ideal MHD system, is classical.\footnote{F.P. thanks Tarek Elgindi for pointing out this example.} We defer the proof to Appendix~\ref{app:growth}.

\begin{proposition}\label{prop:example1}
Consider the 3D ideal MHD system~\eqref{eq:MHDintro1}--\eqref{eq:MHDintro4} with $\mu = \nu = 0$ and periodic boundary conditions (on the torus $\T^3$). There exist smooth initial data $(u_0, B_0)$ of arbitrarily small size in $H^m(\T^3)$, $m > 0$ such that the solution to~\eqref{eq:MHDintro1}--\eqref{eq:MHDintro4} with initial data $(u_0, B_0)$ is global and moreover the following inequality holds:
$$
\|\nabla B \|_{L^\infty(\T^3)} \geq C e^t,
$$
for some positive constant $C$ which depends on initial data.
\end{proposition}

\begin{remark}
The subject of gradient growth in hydrodynamics has been extensively studied, especially in the context of the 2D incompressible Euler equations. In this case, the sharp growth rate of gradients is expected to be double exponential, and it has been shown in the case of domains with boundary~\cite{kiselev}.
\end{remark}

\section{Magnetic relaxation with periodic boundary conditions} \label{sec:relaxation}\label{sec:relax}

In this section, we prove Theorem~\ref{thm:relax1}. We defer the standard setup and definitions to Appendix~\ref{app:setupper}. Recall the Voigt regularized system~\eqref{eq:voigtMHDintro1}--\eqref{eq:voigtMHDintro4} (here, $\mathbb{L} = (- \Delta)^\alpha$): 
\noeqref{eq:MHDV3p}
\begin{align}
&\p_t u + \mathbb{L}^{-1}(u \cdot \nabla u - B \cdot \nabla B) = \nu \mathbb{L}^{-1} \Delta u, \label{eq:MHDV1p} \\
&\p_t  B + \mathbb{L}^{-1}(u \cdot \nabla B - B \cdot \nabla u) = 0, \label{eq:MHDV2p} \\
&(u, B)|_{t = 0} = (u_0, B_0). \label{eq:MHDV3p}
\end{align}

\begin{proof}[Proof of Theorem~\ref{thm:relax1}]
The proof proceeds in three {\bf Steps}. In {\bf Step 1}, we show an integrated a-priori control on the time derivative of $u$ (see inequality~\eqref{eq:goal1}). In {\bf Step 2}, we use~\eqref{eq:goal1} to deduce that, along a sequence of times $t_k \to \infty$, we have the required convergence. Finally, in {\bf Step 3}, we show that the modified magnetic helicity is conserved along the evolution.

\vspace{15pt}

{\bf Step 1.} In this step, we prove inequality~\eqref{eq:goal1}. We claim that there exist positive constants $c_1$ and $c_2$ depending on $\|u_0\|_\alpha, \|B_0 \|_{\alpha} $ and on the parameter $\alpha$, such that the following inequality holds true, for any $t \geq 0$:
\begin{equation}\label{eq:goal1}
\Vert \p_t u(\cdot, t)\Vert^2_{\alpha}+c_1\int_{0}^{t} \|  \p_t \nabla u(\cdot, s) \|_0^2 \de s \leq  c_2.
\end{equation}

In order to prove~\eqref{eq:goal1}, we differentiate~\eqref{eq:MHDV1p} by $\partial_t$. We let $\LL := (-\Delta)^\alpha$:
\begin{equation}\label{eq:dtcomm}
\p_t \LL \p_t u + \mathbb{P}(\p_t u \cdot \nabla u + u \cdot \nabla \p_t u) =\mathbb{P}(\p_t (B \cdot \nabla B ) )  + \nu \Delta \p_t u.
\end{equation}
Hence, upon multiplication of equation~\eqref{eq:dtcomm} by $\p_t u$ and integrating on $\T^3 \times [0,t]$, we have
\begin{equation}\label{eq:dtmain}
\begin{aligned}
&\frac {1} 2 \Vert \p_t u(\cdot, t)  \Vert^2_{\alpha} + \int_{0}^t \int_{\T^3} (\p_t u_i \mathbb{P}( \p_t u_j \p_j u_i))(s,x) \de x\de s\\
& + \nu \int_{0}^t \int_{\T^3} \|\p_t \nabla u\|_0^2 \de x\de s -  \int_{0}^t \int_{\T^3} \p_t u_i\p_t \mathbb{P}( \p_j (B_i B_j) )\de x\de s \leq C.
\end{aligned}
\end{equation}
Integrating $\p_j$ by parts once in space, using the divergence-free condition of $u$, plus the Cauchy--Schwarz inequality, the Poincar\'e inequality~\eqref{eq:poinca}, Morrey's inequality ($\|f\|_{L^6(\T^3)} \leq C \|f\|_{{\dot H}^1(\T^3)}$) and H\"older's inequality, we have
\begin{equation}\label{eq:dtone}
\begin{aligned}
&\Big|\int_{0}^t \int_{\T^3} (\p_t u_i \mathbb{P}(\p_t u_j \p_j u_i))(s,x) \de x\de s\Big| \leq \int_{0}^t \int_{\T^3} |\p_t \nabla u| \, |\mathbb{P}(u\p_t u) | \de x\de s \\
&\leq \frac \nu {10} \Vert \p_t \nabla u\Vert^2_{L^2(0,t; L^2)} + C \int_0^t  \Vert \p_t u (\cdot, s) \Vert^2_{L^4(\T^3)} \Vert  u(\cdot, s)\Vert^2_{L^4(\T^3)} \,\de s\\
&\leq \frac \nu {10} \Vert \p_t \nabla u\Vert^2_{L^2(0,t; L^2)} + C \int_0^t  \Vert \p_t u (\cdot, s) \Vert^2_{\alpha} \Vert  u(\cdot, s)\Vert^2_{1} \,\de s.
\end{aligned}
\end{equation}
Here, we also used that $\alpha \geq 1$.

Concerning the quadratic term in $B$ in~\eqref{eq:dtmain}, we integrate by parts in space to obtain
\begin{equation*}
\begin{aligned}
&\Big| \int_{0}^t \int_{\T^3} \p_t u_i\p_t \mathbb{P}( \p_j (B_i B_j) )\de x\de s \Big| \leq C \int_{0}^t \| \p_t \nabla u\|_0 \, \|\mathbb{P}(B \p_t B)\|_0 \de s \\
& \quad \leq C \int_{0}^t \| \p_t \nabla u\|_0 \, \|B \p_t B\|_0 \de s  \leq C \int_{0}^t \| \p_t \nabla u\|_0 \, \Vert B \Vert_{L^6(\T^3)} \|\p_t B\|_{L^3(\T^3)} \de s.
\end{aligned}
\end{equation*}
Now, \eqref{eq:MHDV2p} implies $\p_t B = \LL^{-1}(\nabla \times (u \times B))$.
This gives, together with the a-priori bound $\|B\|_{L^6(\T^3)} \leq C$:
\begin{equation*}
 \int_{0}^t \| \p_t \nabla u\|_0 \, \Vert B \Vert_{L^\infty} \|\p_t B\|_0 \de s \leq \frac{\nu}{10} \Vert \p_t \nabla u\Vert^2_{L^2(0,t; L^2)} +  C \Vert \LL^{-1} \nabla \times (u \times B)\Vert^2_{L^2(0,t; L^3)}.
\end{equation*}
By $L^p$ elliptic estimates for $\LL$, Sobolev embedding, and the Poincar\'e inequality, we have
\begin{equation*}
\Vert \LL^{-1} \nabla \times (u \times B)\Vert^2_{L^2(0,t; L^3)} \leq C \Vert  u \times B\Vert^2_{L^2(0,t; L^3)} \leq C\Vert B \Vert^2_{L^\infty(0,t; L^6)}\Vert u\Vert^2_{L^2(0,t; L^6)}.
\end{equation*}
Hence we obtain, using Sobolev embedding and the a priori bounds $\|B\|_{L^6(\T^3)} \leq C$, $\|u\|_{L^2(0, \infty; \dot H^1)} \leq C$,
\begin{equation}\label{eq:dttwo}
\begin{aligned}
&\Big| \int_{0}^t \int_{\T^3} \p_t u_i\p_t \mathbb{P}(\p_j (B_i B_j)) \de x\de s \Big| \leq \frac \nu {10} \Vert \p_t \nabla u \Vert^2_{L^2(0,t; L^2)} + C\Vert u\Vert^2_{L^2(0,t; {\dot H}^1)}\\
&\leq \frac \nu {10} \Vert \p_t \nabla u \Vert^2_{L^2(0,t; L^2)} + C.
\end{aligned}
\end{equation}
Combining the above estimates~\eqref{eq:dtmain}, \eqref{eq:dtone}, \eqref{eq:dttwo}, we obtain
\begin{equation*}
\begin{aligned}
&\frac {1} 2 \Vert \p_t u(\cdot, t)  \Vert^2_{\alpha} + \frac \nu 2 \int_{0}^t \int_{\T^3} \|\p_t \nabla u\|_0^2 \de x\de s 
\leq C \int_0^t  \Vert \p_t u (\cdot, s) \Vert^2_{\alpha} \Vert  u(\cdot, s)\Vert^2_{1} \,\de s + C
\end{aligned}
\end{equation*}
We then deduce~\eqref{eq:goal1} by an application of Gr\"onwall's inequality, observing that $ \int_0^t \Vert  u(\cdot, s)\Vert^2_{1} \,\de s \leq C$ uniformly in $t$.

\vspace{20pt}

{\bf Step 2. } In this step, we show the desired convergence along a sequence of times $t_n \to \infty$. We first show two claims.

\vspace{10pt}
{\bf Claim 1: } Inequality~\eqref{eq:goal1} implies that there exists a sequence $\{t_n\}_{n \in \N}$, with $t_n \to \infty$, such that $\Vert \p_t u (\cdot, t_n)\Vert^2_\gamma \to 0$ as $n \to \infty$, for $\gamma = 1$ and all $\gamma < \alpha$.

\vspace{10pt}

\noindent {\emph{Proof of Claim 1:}} Using~\eqref{eq:goal1} we obtain, for a dyadic sequence $t_n$ such that $2^n \leq t_n \leq 2^{n+1}$,  $\|\p_t \nabla u(\cdot, t_n) \|_0^2 \leq C/2^n$. Since $\Vert \p_t u(\cdot, t)\Vert^2_{\alpha} \leq C$, we then have by interpolation that  $\Vert \p_t u(\cdot, t_n)\Vert^2_{\gamma} \to 0$ as $n \to \infty$, for all $\gamma < \alpha$. This concludes the proof of {\bf Claim 1}.

\vspace{10pt}
{\bf Claim 2:} Inequality~\eqref{eq:goal1} implies that $\Vert u(\cdot, t) \Vert_\gamma \to 0$ as $t \to \infty$, for $\gamma = 1$ and all $\gamma < \alpha$.

\vspace{10pt}

\noindent {\emph{Proof of Claim 2:}} We know that the following two inequalities are true:
$$\int_{0}^{\infty} \|\nabla u(\cdot, s)\|_0^2 \de s \leq C, \qquad \int_{0}^{\infty} \| \p_t \nabla u(\cdot, s)\|_0^2 \de s\leq C.$$ From the first inequality we deduce that, along a dyadic sequence $r_n$ such that $2^n \leq r_n \leq 2^{n+1}$, $\|\nabla u(\cdot, r_n)\|_0 \to 0$. Furthermore, for all $t \leq r_n$,
$$
\|\nabla u(\cdot, t)\|_0^2 \leq \Vert \nabla u \Vert_{L^2(t, r_n; L^2)} \, \Vert \p_t \nabla u \Vert_{L^2(t, r_n; L^2)}+\|\nabla u(\cdot, r_n)\|_0^2.
$$
We  first take $n \to \infty$ and then $t \to \infty$ to obtain that $\|\nabla u(\cdot, t)\|_0 \to 0$ as $t \to \infty$. By interpolation we then conclude the proof of {\bf Claim 2}.
\vspace{10pt}

We take the limit of the momentum equation along the sequence $t_n \to \infty$ which was defined in {\bf Claim 1}. We let $u_n(x) := u(x, t_n)$, $B_n(x) := B(x, t_n)$. Due to the bounds proved in {\bf Claim 1} and {\bf Claim 2}, 
\begin{equation*}
\|\p_t u(\cdot, t_n)\|_\gamma \to 0, \qquad \|((-\Delta )^\alpha \p_t u)(t_n, \cdot)\|_{\gamma- 2 \alpha } \to 0,
\end{equation*}
for all $\gamma < \alpha$ and $\gamma =1$. Furthermore, since $\alpha \geq 1$, and since $\|u_n \|_{L^4} \to 0$,
$$\| \mathbb{P}(u_n\cdot \nabla u_n)\|_{-1} = \| \mathbb{P}(\dive (u_n \otimes u_n) )\|_{-1} \to 0.$$
The momentum equation~\eqref{eq:MHDV1p} can be equivalently rewritten as:
\begin{equation*}
    \p_t \mathbb{L} u + \mathbb{P}(u \cdot \nabla u - B \cdot \nabla B) = \nu \Delta u.
\end{equation*}
The above estimates, in conjunction with the convergence properties of $B_n$, in particular imply that every term in the above equation converges to zero in $\dot H^{-1}$, except for the quadratic term in $B$. Therefore, $B_\infty$ satisfies the steady equations in the sense of distributions:
$$
\mathbb{P}(B_\infty \cdot \nabla  B_\infty) = 0, \qquad \dive B = 0.
$$
Since $B_\infty \cdot \nabla  B_\infty$ is, in fact, in $L^2$, the equations are satisfied in the strong $L^2$ sense. This in particular implies that $B$ is a strong solution to the steady 3D Euler system~\eqref{eq:euler3dthm}, concluding {\bf Step 2} of the proof.

 \vspace{20pt}
 
{\bf Step 3. } In view of Lemma~\ref{lem:magneticpot}, there exists a magnetic potential $\Psi$ such that, in particular, $\nabla \times \Psi = B$. We compute the time derivative of the modified magnetic helicity:
\begin{equation*}
\begin{aligned}
&\frac{\text{d}}{\text{d}t} H[B](t) =  (\p_t  B,  \LL \Psi)  +  (B, \p_t \LL \Psi) =(\p_t  \LL B,   \Psi)  +  (B, \p_t \LL \Psi)  .
\end{aligned}
\end{equation*}
By Lemma~\ref{lem:magneticpot}, $\Psi$ satisfies the equations $\p_t \LL\Psi = u \times B$ and $\nabla \times \Psi = B$. Since $\LL$ is self-adjoint with respect to the $L^2$ inner product, we have
\begin{equation*}
\begin{aligned}
&\frac{\text{d}}{\text{d}t} H[B](t)  = \int_{\T^3}  \p_t \LL B \cdot \Psi \de x  +\int_{\T^3}  B \cdot \p_t \LL \Psi \de x\\
& \qquad = \int_{\T^3}  (\nabla \times (u \times B) )\cdot \Psi \de x + \int_{\T^3}   B \cdot (u \times B) \de x\\
& \qquad = -\int_{\T^3}  (u \times B) \cdot (\nabla \times \Psi) \de x = 0.
\end{aligned}
\end{equation*}
This concludes the proof of the theorem.
\end{proof}

\section{The case of bounded domains}\label{sec:bdd}

In this section, we show Theorem~\ref{thm:bdd}. We first recall a few facts about the analysis of the Stokes operator on bounded domains.

\subsection{Bounded domains: setup}\label{app:setupbdd}

Let $\Omega \subset \R^3$ be a bounded domain with smooth boundary.  For $m \geq 0$, $m \in \N$, we define the spaces $H^{m}(\Omega), H^{m}_0(\Omega)$ as in~\cite{boyerfabrie} (Definition III.2.1), and we define $H^{m}(\Omega)$ for $m < 0$, $m \in \N$ by duality, as in~\cite{boyerfabrie} (Section III.2.6). 

We consider the usual Laplacian on $\Omega$ with zero Dirichlet boundary conditions, it well known that there exists a complete eigenbasis $\{v_k\}_{k \in \N}$, such that each $v_k \in H^1_0(\Omega)$, with associated eigenvalues $\eta_k$.
The above definitions extend to fractional spaces as follows. Let $\gamma \in \R$, $\gamma \geq 0$. We set, for $p > 1$,
\begin{align}
&\|f\|^2_{H_0^\gamma(\Omega)} := \sum_{k \geq 0} \eta^{\gamma}_k (f, v_k)^2,\\
&\|f\|^2_{H^\gamma(\Omega)} := \|f\|^2_{H_0^\gamma(\Omega)} + \|f\|^2_{L^2(\Omega)},\label{eq:norm2}\\
&\|f\|_{L^p(\Omega)} := \Big(\int_\Omega |f|^p \, dx \Big)^{\frac1p}.
\end{align}
Here, the bracket $(f,g)$ denotes the standard $L^2(\Omega)$ inner product. The spaces with negative index $\gamma < 0$ are defined by duality.

We now look at vector-valued spaces. We recall the following classical definition (see~\cite{boyerfabrie}, IV.3.3):
\begin{definition}
We define the spaces of divergence-free vector fields as follows:
\begin{align}
\mathcal{V} := \{\varphi \in (\mathcal{C}^\infty_0(\Omega))^3, \mbox{div } \varphi = 0\}. 
\end{align}
We moreover define $H$ as the closure of $\mathcal{V}$ in $(L^2(\Omega))^3$, and $V$ as the closure of $\mathcal{V}$ in $(H^1_0(\Omega))^3$.
\end{definition}

\begin{remark}
We have the following characterization
\begin{align}
V = \{v \in (H^1_0(\Omega))^3, \mbox{div } v = 0\}.
\end{align}
\end{remark}
We consider the Stokes operator $A: D(A) \to H$ whose action on $v \in D(A) = V \cap (H^2(\Omega))^3$ is given by
\begin{equation}
Av := \mathbb{P} (-\Delta v).
\end{equation}
Here, the definition of Leray projection is as in \cite{boyerfabrie}, Definition IV.3.6.

We consider an $L^2$ orthonormal basis of eigenfunctions of $A$, $\{w_k\}_{k \in \N} \in D(A)$, along with the associated eigenvalues $\lambda_k$. For $\alpha \in \R$, the fractional powers of $A$ are defined as
\begin{equation}
A^{\alpha} v := \sum_{k\geq 0} \lambda^\alpha_k (v, w_k) w_k
\end{equation}
We denote
\begin{align}
&D(A^{\alpha/2}) := \{v \in H: A^{\alpha/2} v \in H\},\\
&\|v\|^2_{\alpha} := \sum_{k\geq 0} \lambda^\alpha_k (v, w_k)^2. \label{eq:normalpha}
\end{align}

\begin{remark}
The inclusion:
\begin{equation}
D(A^{\alpha/2}) \subset V \cap (H^{\alpha}(\Omega))^3
\end{equation}
holds true.
\end{remark}

We set $\LL := A^\alpha$. We also set $\LL^{-1}$ to be the inverse of the operator $\LL$. $\LL^{-1}$ maps $D(A^{\gamma/2})$ to $D(A^{\frac \gamma 2 + \alpha})$, for all $\gamma \in \R$ with $\gamma \geq -1$. 
We have the following Lemma.
\begin{lemma}\label{lem:elliptic}
For all $v \in D(A^{\gamma/2})$, with $\gamma \geq -1$, and for $\kappa > 0$ the following holds:
\begin{equation}
\begin{aligned}
&\|A^{-\kappa}v\|_{H^{\gamma + 2 \kappa}} \leq C \|v\|_{\gamma}.
\end{aligned}
\end{equation}
\end{lemma}

\begin{remark}
Note that the above inequality is between the norm defined in~\eqref{eq:norm2} (or its dual, for negative exponents) and the norm defined in~\eqref{eq:normalpha}.
\end{remark}

We moreover recall the following Proposition (Proposition IV.5.12 in~\cite{boyerfabrie}).
\begin{proposition}
We have the following Poincar\'e inequalities:
\begin{align}
&\|u\|^2_H \leq C \|\nabla u\|^2_{L^2} \quad \forall u \in V,\label{eq:poinca1}\\
&\|\nabla u\|^2_{L^2} \leq C \|A u\|^2_{L^2} \quad \forall u \in D(A).\label{eq:poinca2}
\end{align}
\end{proposition}

We also recall the notion of trace from~\cite{boyerfabrie} (see Theorem III.2.19 and Definition III.2.20). In particular we recall that $\gamma_0$ is the trace operator which maps $W^{1, p}(\Omega)$ to $W^{1- \frac 1p, p}(\p \Omega)$.

\subsection{Proof of Theorem~\ref{thm:bdd}}
Recall the Voigt--MHD system:
\noeqref{eq:MHDVbd2}
\begin{align}
& \p_t u + A^{-\alpha}\mathbb{P}(u \cdot \nabla u - B \cdot \nabla B ) + A^{1-\alpha} u=0, \label{eq:MHDVbdd1}\\
& \p_t B + A^{-\alpha}\mathbb{P} (u \cdot \nabla B - B \cdot \nabla u) = 0,\label{eq:MHDVbd2}\\
&(u, B)|_{t = 0} = (u_0, B_0).\label{eq:MHDVbdd3}
\end{align}

\begin{proof}[Proof of Theorem~\ref{thm:bdd}]
We use Galerkin approximation. Let us consider $\mathbb{P}_k$ to be the orthogonal projector onto the first $k$ eigenfunctions of $A$. Let $W_k := \text{span }(w_1, \ldots, w_k)$. Using this, we first define the following Galerkin approximation of system~\eqref{eq:MHDVbdd1}--\eqref{eq:MHDVbdd3}:
\begin{align}
&\p_t \LL u_k + \mathbb{P}_k (u_k\cdot\nabla u_k - B_k \cdot \nabla B_k) + A u_k=0, \label{eq:galerkin1}\\
&\p_t \LL B_k + \mathbb{P}_k (u_k\cdot\nabla B_k - B_k \cdot \nabla u_k) = 0, \label{eq:galerkin2}\\
&(u_k, B_k)|_{t=0} = (\mathbb{P}_k u_0, \mathbb{P}_k B_0). \label{eq:galerkin3}
\end{align}

Here, $u_k, B_k \in W_k$. This reduces the system~\eqref{eq:MHDVbdd1}--\eqref{eq:MHDVbdd3} to a system of $2k$ ODEs.

Since $\mathbb{P}_k$ is orthogonal, we immediately deduce the following conservation law for the reduced system:
\begin{equation}\label{eq:cons-k}
\frac 12 \p_t ( \|u_k\|^2_{\alpha} + \|B_k\|^2_{\alpha}) + \|u_k\|^2_{(H^1_0(\Omega))^3} = 0.
\end{equation}

\vspace{5pt}

\noindent {\bf Step 1.} \emph{Global well-posedness of the projected system.} Picard's theorem, in conjunction with the bound given by~\eqref{eq:cons-k}, gives that solutions to~\eqref{eq:galerkin1}--\eqref{eq:galerkin3} are global.

\vspace{10pt}

\noindent {\bf Step 2.} \emph{Propagation of regularity.} We have that $(\mathbb{P}_k(u_0), \mathbb{P}_k(B_0)) \in D(A^{\beta/2})$, for all $\beta > \alpha$. We apply $A^{\beta - \alpha}$ to both equations~\eqref{eq:galerkin1}, \eqref{eq:galerkin2}. We multiply~\eqref{eq:galerkin1} by $A^{\beta - \alpha} u_k$, and we multiply~\eqref{eq:galerkin2} by $A^{\beta - \alpha} B_k$, and add them. Using Sobolev embedding, the Poincar\'e inequality~\eqref{eq:poinca2}, and the uniform control given by~\eqref{eq:cons-k}, we deduce that there is a constant $C(k, \beta, u_0, B_0)$ such that the following bound holds true for all $t \geq 0$ and all $k \geq 0$, $k \in \N$:
\begin{equation}
 \|u_k\|^2_{\beta} + \|B_k\|^2_{\beta} \leq C(k, \beta, u_0, B_0) e^{C(k, \beta, u_0, B_0)t}.
\end{equation}
This non-uniform a-priori bound will be needed to justify the application of Leibnitz rule to products of functions in the next step.

\vspace{10pt}

\noindent {\bf Step 3.} \emph{Time integrability of $\p_t u_k$.} The goal of this step is to show the following time integrability estimate:
\begin{equation}\label{eq:timeint}
\int_0^t (\|\p_t u_k(s)\|^2_{(L^2(\Omega))^3} + \frac c2 \|\p_t \nabla u_k(s)\|^2_{(L^2(\Omega))^3})  ds \leq C(u_0, B_0).
\end{equation}
Here, $C(u_0, B_0)$ is a constant which only depends on initial data, and $c > 0$ is a positive constant which does not depend on initial data.

To that end, we multiply equation~\eqref{eq:galerkin1} by $\p_t u_k$. We obtain, since $\mathbb{P}_k$ is orthogonal, denoting by $( \cdot, \cdot )_2$ the usual $L^2$ inner product, 
\begin{align*}
&\frac 12 \p_t \|A^{\frac 12}u_k\|^2_{(L^2(\Omega))^3} + (\p_t A^{\alpha/2}u_k, \p_t A^{\alpha/2}u_k )_2 + ( \p_t u_k, u_k \cdot \nabla u_k)_2 - ( \p_t u_k, B_k \cdot \nabla B_k)_2 =0.
\end{align*}
We have that $(\p_t A^{\alpha/2}u_k, \p_t A^{\alpha/2}u_k )_2 \geq c \|\p_t \nabla u_k\|_{L^2}$ for some constant $c>0$ (from the definition of $A$ and from inequality~\eqref{eq:poinca2}), so that
\begin{equation}\label{eq:est1}
\frac 12 \p_t \|A^{\frac 12}u_k\|^2_{(L^2(\Omega))^3}+\frac c 2 \|\p_t \nabla u_k\|_{(L^2(\Omega))^3}^2 + ( \p_t \nabla u_k, B_k \otimes B_k)_2 \leq  C\|u_k\|^2_{L^4(\Omega)}.
\end{equation}
In the above inequality, we used H\"older's inequality, integration by parts, the divergence free condition of $u_k$ and $B_k$ and the fact that, by the trace theorem, $u_k, B_k =0$ in $(H^{\frac 12}(\partial \Omega))^3$.
We then have that
\begin{equation}\label{eq:est2}
\begin{aligned}
 ( \p_t \nabla u_k, B_k \otimes B_k)_2 = \p_t ( \nabla u_k, B_k \otimes B_k)_2 -  ( \nabla u_k, \p_t B_k \otimes B_k)_2 -  ( \nabla u_k, B_k \otimes \p_t B_k)_2.
\end{aligned}
\end{equation}

Then, by H\"older's inequality,
\begin{equation}\label{eq:est3}
\begin{aligned}
&( \nabla u_k, \p_t B_k \otimes B_k)_2 = ( \nabla u_k,\LL^{-1}(\mathbb{P}_k(\nabla \times( u_k \times  B_k)))\otimes B_k)_2\\
&\leq C \|\nabla u_k\|_{L^2(\Omega)}\| \LL^{-1}(\mathbb{P}_k(\nabla \times( u_k \times  B_k))\|_{(L^3(\Omega))^3} \| B_k\|_{(L^6(\Omega))^3}\\
&\leq C \|\nabla u_k\|_{L^2(\Omega)}\|A^{\frac 14} \LL^{-1}(\mathbb{P}_k(\nabla \times( u_k \times  B_k))\|_{(L^2(\Omega))^3} \| B_k\|_{(L^6(\Omega))^3}\\
&\leq C \|\nabla u_k\|_{L^2(\Omega)}\|u_k\|_{(L^4(\Omega))^3} \| B_k\|_{(L^4(\Omega))^3}\| B_k\|_{(L^6(\Omega))^3}\\
&\leq C(u_0, B_0)\|\nabla u_k\|^2_{(L^2(\Omega))^3}.
\end{aligned}
\end{equation}
In the last inequality, $C(u_0, B_0)$ is a constant depending only on initial data, and moreover we used the Leibnitz rule, H\"older's inequality, Lemma~\ref{lem:elliptic}, Sobolev embedding, Poincar\'e's inequality~\eqref{eq:poinca1} and the conservation law~\eqref{eq:cons-k}.

We combine~\eqref{eq:est1}, \eqref{eq:est2}, \eqref{eq:est3} and use the Poincar\'e inequality~\eqref{eq:poinca2} in order to obtain
\begin{equation}
\begin{aligned}
&\frac 12 \p_t \big(\|A^{\frac 12}u_k\|^2_{(L^2(\Omega))^3}+ ( \nabla u_k, B_k \otimes B_k)_2\big) +\frac c 2 \|\p_t \nabla u_k\|_{(L^2(\Omega))^3}^2  \leq  C(u_0, B_0) \|\nabla u_k\|^2_{(L^2(\Omega))^3}.
\end{aligned}
\end{equation}
Integrating in time, and using again the conservation law~\eqref{eq:cons-k} on the boundary term at time $t$, and on the term on the RHS, plus the Poincar\'e inequality, we obtain, for all times $t \geq0$,
\begin{equation}
\int_0^t (\|\p_t u_k(s)\|^2_{(L^2(\Omega))^3} + \frac c2 \|\p_t \nabla u_k(s)\|^2_{(L^2(\Omega))^3})  ds \leq C(u_0, B_0).
\end{equation}

\vspace{10pt}

\noindent {\bf Step 4.} \emph{Taking the limit $k \to \infty$.} We notice that the embedding $D(A^{\gamma_1/2})\hookrightarrow D(A^{\gamma_2/2})$ is compact, whenever $\gamma_1 > \gamma_2$. Due to the conservation~\eqref{eq:cons-k}, the Aubin--Lions lemma, and a diagonal argument, we can take a subsequential limit in $k$ and obtain the limiting objects $(u,B)$ such that\footnote{We say that $f \in D((-\Delta)^{\frac{\beta-}{2}})$ if for all sufficiently small $\eta > 0$, $f \in D((-\Delta)^\frac{\beta - \eta}{2})$.}
$$
(u, B) \in L^\infty(0, T; D(A^{\frac{\alpha-}{2}})) \times L^\infty(0, T; D(A^{\frac{\alpha-}{2}}))
$$
for all $T \geq 0$.
Moreover, the convergence $u_k \to u$ and $B_k \to B$ is strong in the above topology. In addition, we have that, in the limit, the uniform bounds
\begin{align}\label{eq:cons-fin1}
&\frac 12   \|u(t)\|^2_{\alpha} + \frac 12   \|B(t)\|^2_{\alpha} + \int_0^t \|u(s)\|^2_{(H^1_0(\Omega))^3} ds 
\leq C(u_0, B_0),\\
&\int_0^t (\|\p_t u(s)\|^2_{(L^2(\Omega))^3} + \frac c2 \|\p_t \nabla u(s)\|^2_{(L^2(\Omega))^3})  ds \leq C(u_0, B_0).\label{eq:cons-fin2}
\end{align}
hold true for all $t \geq 0$. Moreover, $(u, B)$ satisfy system~\eqref{eq:MHDVbdd1}--\eqref{eq:MHDVbdd3} in the strong sense.

\vspace{10pt}

\noindent {\bf Step 5.} \emph{Proof of relaxation.} Once we have the bounds~\eqref{eq:cons-fin1},~\eqref{eq:cons-fin2}, we essentially repeat the same argument as in the $\T^3$ case, keeping in mind that this time we have a slightly different bound for $\p_t u$, in~\eqref{eq:cons-fin2}.

First, the bound~\eqref{eq:cons-fin2} implies that there is a sequence $t_i \in \R$, and a uniform constant $C_1>0$, depending only on the initial data $(u_0, B_0)$ such that 
\begin{align}
&\|\p_t \nabla u(t_i) \|_{(L^2(\Omega))^3} \leq \frac{C_1}{t_i}, \qquad t_i \in [2^{i}, 2^{i+1}].\label{eq:decay1}
\end{align}
Moreover, from~\eqref{eq:cons-fin1}, and using~\eqref{eq:decay1}, and the Poincar\'e inequality we deduce that 
\begin{equation*}
\|u(t)\|_{(L^2(\Omega))^3} \to 0 \qquad \text{as} \quad t \to \infty,
\end{equation*}
and by interpolation using the bound~\eqref{eq:cons-k}, we have also that
\begin{equation}\label{eq:decaygamma}
\|u(t)\|_{\gamma} \to 0 \qquad \text{as} \quad t \to \infty,
\end{equation}
for all $\gamma \in [0, \alpha)$.

Using the bounds~\eqref{eq:decay1},~\eqref{eq:decaygamma}, in conjunction with~\eqref{eq:MHDVbdd1}, we have that 
$$\mathbb{P}(B(t_i) \cdot \nabla B(t_i)) \to 0$$ 
weakly in $(L^2(\Omega))^3$. We note that 
$\mathbb{P}(B_\infty \cdot   \nabla B_\infty) \in L^2(\Omega),$
hence $B_\infty$ is a strong solution.

\vspace{10pt} 

\noindent {\bf Step 6.} \emph{$B_\infty$ is not a Beltrami flow.} We use an elementary argument (see also~\cite{vainshtein}). Suppose that we have a field $B \in D(A) \subset H_0^1(\Omega)$, such that $(\nabla \times B) \times B= 0$, $\dive B = 0$.

Then, $B \cdot \nabla B - \frac 12 \nabla |B|^2 = 0$, which implies, by taking the divergence, $\p_i \p_j(B_i B_j) - \frac 12 \Delta|B|^2 = 0$ (repeated indices are summed). We then multiply this equation by $|x|^2$ and integrate on $\Omega$. We integrate by parts twice. The boundary terms vanish due to the condition $B \in H_0^1(\Omega)$, and the bulk term yields $\|B\|_{L^2(\Omega)} = 0$.

\vspace{10pt}

\noindent {\bf Step 7.} \emph{The limiting $B_\infty$ is nontrivial in the case $\alpha = 1$.} For $\Psi \in D(A^{(\alpha +1)/2})$ sufficiently regular, denote the function
\begin{equation}
H[B, \Psi] := \int_\Omega \LL B \cdot \Psi dx.
\end{equation}
We then derive conditions on $\Psi$ such that the above expression is conserved.
We have:
\begin{equation}
\begin{aligned}
&\p_t \int_\Omega  \LL B \cdot \Psi dx =  \int_\Omega  \p_t  \LL  B \cdot \Psi dx +   \int_\Omega   \LL  B \cdot \p_t \Psi dx =  \int_\Omega \mathbb{P} (\nabla \times (u \times B) )\cdot \Psi dx +   \int_\Omega  B \cdot \p_t \LL  \Psi dx\\
& = \int_\Omega (u \times B) \cdot (\nabla \times \Psi) dx +   \int_\Omega  B \cdot \p_t \LL  \Psi dx =  \int_\Omega ((\nabla \times \Psi) \times u) \cdot B dx +   \int_\Omega  B \cdot \p_t \LL  \Psi dx\\
&=  \int_\Omega \mathbb{P}( (\nabla \times \Psi) \times u) \cdot B dx +   \int_\Omega  B \cdot \p_t \LL  \Psi dx.
\end{aligned}
\end{equation}
Here, we used integration by parts and the fact that $(V_1 \times V_2) \cdot V_3 = (V_3 \times V_1 ) \cdot V_2$. We then see that, upon setting
\begin{equation}
\p_t \LL \Psi  +\mathbb{P}((\nabla \times \Psi) \times u) =0,
\end{equation}
$H[B, \Psi]$ is conserved in time.

We set $\Psi$ to solve the following initial value problem:
\begin{equation}
\begin{aligned}
&\p_t \LL \Psi + \mathbb{P}( (\nabla \times \Psi) \times u )= 0,\\
&\Psi(t=0) = \Psi_0,
\end{aligned}
\end{equation}
where we chose $\Psi_0$ such that $\nabla \times \Psi_0 = B_0$. 
In this case, then, we have that $H[B, \Psi] = H[B_0, \Psi_0]  \neq 0$ is conserved.

Moreover, since $\alpha =1$, $- \mathbb{P}( \p_t \Delta \Psi + (\nabla \times \Psi) \times u )= 0,$
which implies $-\p_t \Delta \Psi + (\nabla \times \Psi) \times u = \nabla q,$ for some sufficiently regular function $q$. We then take the curl of the above equation, to obtain
$$
-\p_t \Delta( \nabla \times \Psi) + \nabla \times((\nabla \times \Psi) \times u) = 0.
$$
From this, and from the fact that $\nabla \times \Psi_0 = B_0$, we conclude that $\nabla \times \Psi = B$. This implies that $B_\infty$ is nontrivial.
\end{proof}

\appendix

\section{Periodic boundary conditions: setup}\label{app:setupper}

Let $\T^3$ be the three-dimensional torus.

\subsection{Norms and spaces}
 We work with the spaces $H^s(\T^3)$ and ${\dot H^s}(\T^3)$, which we define as follows.
\begin{definition}
Let $f \in \mathscr{P}'$, where $\mathscr{P}'$ is the space of periodic distributions on the $3$-dimensional torus $\T^3$. Then, consider the Fourier coefficients $\hat f(k)$, where $k \in \Z^3$. We say that $f \in H^s(\T^3)$ if the Fourier coefficients of $f$ satisfy:
\begin{equation}
\Vert f \Vert_{H^s(\T^3)} := \sum_{k \in \Z^3} (1+|k|^2)^{\frac s 2} |\hat f(k)|^2 < \infty. 
\end{equation}
 We also say that $f \in {\dot H^s}(\T^3)$, if the Fourier coefficients of $f$ satisfy:
\begin{equation}
 \hat f(0,0,0) = 0, \qquad \Vert f \Vert_{{\dot H^s}(\T^3)} := \sum_{\substack{k \in \Z^3\\ k \neq (0,0,0)}} (1+|k|^2)^{\frac s 2} |\hat f(k)|^2 < \infty. 
\end{equation}
We moreover denote
\begin{equation}
{\dot{L}}^2(\mathbb{T}^3) := {\dot{H}}^0(\mathbb{T}^3).
\end{equation}
\end{definition}

\begin{remark}
In the remainder of section~\ref{sec:relax} (and \emph{only} restricted to this section), we denote
$$
{\dot H^s} := {\dot H^s}(\mathbb{T}^3), \qquad H^s:= H^s(\mathbb{T}^3).
$$
\end{remark}

\begin{remark}
For conciseness, we denote $\Vert f \Vert_{s} := \Vert f \Vert_{{\dot H^s}}$, the \emph{homogeneous} Sobolev norm.
\end{remark}

 If $f, g \in {\dot L}^2(\mathbb{T}^3)$, we let
\begin{equation}
(f,g):= \int_{\T^3} fg \de x.
\end{equation}
the usual $L^2$ inner product. 

The dual of ${\dot H}^1$ is naturally identified with ${\dot H}^{-1}$. If $f \in {\dot H}^{-1}$ and $g \in {\dot H}^1$, we let $\langle f, g \rangle$ be the dual pairing between $f$ and $g$. Note that, if $f\in {\dot L}^2$, then $\langle f, g \rangle = (f,g)$. 

These scalar spaces extend in a straightforward manner to their vectorial counterparts. We shall use a superscript to denote vectorial spaces, i.e. if $v$ is a 3 dimensional vector field whose components lie in a space $W$, then we write $v \in W^3$.

\begin{definition}
We denote $H:= \{v \in L_0^2(\T^3): \text{div} \ v = 0 \}$, and $V:= \{v \in {\dot H}^1(\T^3): \text{div} \, v = 0 \}$.
\end{definition}

\begin{definition}\label{def:vspace}
We define the space $D((-\Delta)^{s/2})$ for $s \in \R$ as the domain of $(-\Delta)^{s/2}$:
\begin{equation}
D((-\Delta)^{s/2}) := \{v \in ({\dot H^s})^3: \text{ div } v = 0\}.
\end{equation}
Without ambiguity, we denote the homogeneous Sobolev $s$-norm on $D((-\Delta)^{s/2})$ also by $\| \cdot \|_s$. 
\end{definition}

\subsection{Leray projection, the fractional Laplacian, and classical estimates}

Let $\mathbb{P}: L_0^2 \to H$ be the Leray projector. In terms of Fourier coefficients, it can be expressed as
\begin{equation*}
(\widehat{\mathbb{P} f})_k := \Big(\text{Id}_3 - \frac{k \otimes k}{|k|^2} \Big)(\hat f)_k \quad \text{for } k \neq 0, \qquad (\widehat{\mathbb{P} f})_0 := 0.
\end{equation*}
Here, $\mathrm{Id}_3$ is the $3\times3$ identity matrix.

If $v, w \in V$, we define
\begin{equation}\label{eq:bformdef}
\boldb(v,w) := \mathbb{P}(v \cdot \nabla w).
\end{equation}
Let us also recall the following inequality, proved in~\cite{constbook}. There exists a constant $C > 0$ such that, if $u,v,w \in V$ and are all divergence-free,
\begin{equation}\label{eq:estb}
\langle \boldb(u,v), w \rangle \leq C \|u\|_0^{\frac 12} \Vert u \Vert_1^{\frac 12} \Vert v \Vert_1 \Vert w\Vert_1.
\end{equation}
We also recall that elements $u \in {\dot H}^1$ enjoy the following Poincar\'e inequality:
\begin{equation}\label{eq:poinca}
\| u \|_0 \leq \lambda \Vert u \Vert_1,
\end{equation}
where $\lambda > 0$ is the Poincar\'e constant.

We define the fractional Laplacian as the operator $(-\Delta)^\alpha: {\dot H^s} \to \dot H^{s-2\alpha}$ whose action on the Fourier coefficients is as follows:
\begin{equation}
(\widehat{(-\Delta)^\alpha f})_k = |k|^{2 \alpha} \hat f_k, \quad \text{for } k \neq 0, \qquad \hat f_0 = 0.
\end{equation}
We define
\begin{equation}
\LL :=  (-\Delta)^\alpha.
\end{equation}
We have the following lemma:
\begin{lemma}\label{lem:ellipticL}
Let $\LL :=  (-\Delta)^\alpha$, and assume that $\alpha \geq 1$. Then, $\LL$ maps ${\dot H^s}$ onto $\dot H^{s-2\alpha}$, is injective and has a well defined inverse $\LL^{-1}: \dot H^{s-2\alpha}\to {\dot H^s}$. Furthermore, the composition operators $\Delta \LL^{-1} =  \LL^{-1} \Delta$ are well defined, and are both bounded operators from ${\dot H^s}$ to $\dot H^{2\alpha - 2 +s}$. In particular, there exists a constant $C > 0$ such that, for all $v \in {\dot H^s}$,
\begin{equation}
\Vert \LL^{-1}\Delta v \Vert_{2\alpha - 2 + s} = \Vert \Delta \LL^{-1}v \Vert_{2\alpha -2 +s } \leq C \Vert v\Vert_{s}.
\end{equation}
\end{lemma}

\begin{remark}
The lemma above holds true replacing all instances of $\dot H^\gamma$ with $D((-\Delta)^{\gamma/2})$.
\end{remark}

\begin{proof}[Proof of lemma~\ref{lem:ellipticL}]
The proof follows from the spectral characterization of the operator $(-\Delta)^\alpha$.
\end{proof}

\begin{remark}
By the Fourier characterization of $(-\Delta)^\alpha$ and of $\mathbb{P}$, it is evident that both these operators commute with partial derivatives. Furthermore, we have that
$$
[\mathbb{P}, \LL] =0, \qquad [\mathbb{P}, (-\Delta)^\alpha] =0.
$$
\end{remark}

We also recall the definition of time-dependent spaces. Let $(X, \Vert \cdot \Vert_X)$ be a Banach space, and let $T > 0$. We say that a function $f: [0,T] \to X$ is such that $f \in L^\infty(0, T; X)$ if the following holds: 
$$\sup_{t \in [0,T]  } \Vert f(t) \Vert_X < \infty.$$
Let $f: [0,T] \to X$ be Bochner integrable, and define $f'$ as the weak time derivative of $f$. We then say that $f \in \mathcal{C}^1(0, T; X)$ if there holds 
$$\sup_{t \in [0,T]  } (\Vert f(t) \Vert_X + \Vert f'(t) \Vert_X) < \infty.$$

\subsection{The viscous Voigt--MHD system}\label{sec:voigtsys}

We focus our attention on the Voigt regularized system~\eqref{eq:voigtMHDintro1}--\eqref{eq:voigtMHDintro4} (recall that $\mathbb{L} = (- \Delta)^\alpha$): 
\noeqref{eq:MHDV2}
\begin{align}
&\p_t u + \mathbb{L}^{-1}(u \cdot \nabla u - B \cdot \nabla B) = \nu \mathbb{L}^{-1} \Delta u, \label{eq:MHDV1} \\
&\p_t  B + \mathbb{L}^{-1}(u \cdot \nabla B - B \cdot \nabla u) = 0, \label{eq:MHDV2} \\
&(u, B)|_{t = 0} = (u_0, B_0). \label{eq:MHDV3}
\end{align}

Here, $u(x,t)$ and $B(x,t)$ are three-dimensional vector fields depending on position $x \in \T^3$ (the three-dimensional flat torus) and time $t$. Furthermore $\nu >0$ is a fixed parameter, and $q(x,t)$ is the pressure term.

Let $\alpha \geq 1$. We set up initial data $u_0$ and $B_0$ such that $u_0, B_0 \in D((-\Delta)^{\alpha/2})$ (see Definition~\ref{def:vspace}). In particular, $u_0$, $B_0$ satisfy:
\begin{equation}\label{eq:initcond}
\dive  u_0 = 0, \qquad \dive B_0 = 0,
\end{equation}
and moreover
\begin{equation}\label{eq:meanz}
	\int_{\T^3} B_0 \de x = 0, \qquad \int_{\T^3} u_0 \de x = 0.
\end{equation}

\subsection{Well-posedness, global existence and regularity}\label{sec:globalvoigt}

We show that the Voigt--MHD system~\eqref{eq:MHDV1}--\eqref{eq:MHDV3} is locally well posed for strong solutions, and it moreover admits global solutions for large initial data. In the Voigt case, strong regularization easily implies well-posedness. Note that the issue is generally more involved in the case of non-regularized systems: for a discussion of local ill-posedness for a wide range of hydrodynamical systems we refer the reader to~\cite{elgindimasmoudi}.

We first prove a local existence statement with a suitable continuation criterion.

\begin{proposition}[Local existence of solutions to~\eqref{eq:MHDV1}--\eqref{eq:MHDV3}]\label{prop:localMHD}

Let $(u_0, B_0)$ in $D((-\Delta)^{\alpha/2})$ with $\alpha \geq 1$, and let $u_0$ and $B_0$ satisfy the divergence-free condition~\eqref{eq:initcond} and the mean zero condition~\eqref{eq:meanz}. Then, there exists a time $T > 0$ and $u, B \in L^\infty(0,T; D((-\Delta)^{\alpha/2})$ which solve~\eqref{eq:MHDV1}--\eqref{eq:MHDV3} in the strong sense, and such that
\begin{equation*}
(u, B)|_{t= 0} = (u_0, B_0).
\end{equation*}
Furthermore, if $T_*$ is the maximal time of existence, we necessarily have either
\begin{equation}
\limsup_{t \to T_*} \Vert u(\cdot, t) \Vert_{\alpha} = \infty, \quad or \qquad \limsup_{t \to T_*} \Vert B(\cdot, t) \Vert_{\alpha} = \infty.
\end{equation}
\end{proposition}

\begin{proof}[Proof of Proposition~\ref{prop:localMHD}]
The proof follows by Picard iteration carried out in the space 
$$(u, B) \in L^\infty(0,T; D((-\Delta)^{\alpha/2})).$$ We follow the approach in~\cite{lariostiti} (Theorem 6.1). 

We consider the following evolution equations, which are equivalent to system~\eqref{eq:MHDV1}--\eqref{eq:MHDV3}:
\begin{equation}
\begin{aligned}
&\p_t \LL u = \mathbb{P}(B \cdot \nabla B - u \cdot \nabla u) + \nu \Delta u, \\
&\p_t \LL B  = \mathbb{P} (- u \cdot  \nabla B + B \cdot \nabla u).
\end{aligned}
\end{equation}
We then let $v = \LL u$, $Z = \LL B$, and we rewrite the system as follows:
\begin{equation}
\begin{aligned}
&\p_t v = \mathbb{P}(\LL^{-1}Z \cdot \nabla \LL^{-1}Z - \LL^{-1}v \cdot \nabla \LL^{-1}v) + \nu \Delta \LL^{-1}v =: N_1(v,Z), \\
&\p_t Z  = \mathbb{P} (- (\LL^{-1}v)\cdot  \nabla (\LL^{-1}Z) + (\LL^{-1}Z) \cdot \nabla (\LL^{-1}v)) =: N_2(v,Z).
\end{aligned}
\end{equation}
The idea is to show that $N_1$ and $N_2$ are Lipschitz mappings from $D((-\Delta)^{-\alpha/2}))$ to itself, and then the Picard--Lindel\"of theorem will apply. We start from $N_1$, using the conventions $v_i = \LL u_i$, $Z_i = \LL B_i$:
\begin{align*}
& \Vert N_1(v_1, Z_1) - N_1(v_2, Z_2)\Vert_{-\alpha} \\
&\leq \Vert \boldb(u_1 - u_2, u_2) + \boldb(u_1 , u_1 -u_2)  \Vert_{-\alpha} + \Vert \boldb(B_1 - B_2, B_2) + \boldb(B_1 , B_1 -B_2)  \Vert_{-\alpha} \\
&\qquad + \nu \Vert \Delta (u_1 - u_2) \Vert_{-\alpha} \\
&\leq C\|u_1 -u_2\|_0 \Vert u_1 - u_2\Vert_1 \Vert u_2 \Vert_1 + C \|u_1\|_0 \Vert u_1\Vert_1 \Vert u_1- u_2 \Vert_1 \\
&\quad  +C\|B_1 -B_2\|_0 \Vert B_1 - B_2\Vert_1 \Vert B_2 \Vert_1 + C \|B_1\|_0 \Vert B_1\Vert_1 \Vert B_1- B_2 \Vert_1 +\nu \Vert \Delta (u_1 - u_2) \Vert_{-\alpha}\\
&\leq C\lambda (\Vert v_1 \Vert_{1-2\alpha} +  \Vert v_2 \Vert_{1-2\alpha} ) \Vert v_1 - v_2  \Vert_{1-2\alpha} + C \lambda (\Vert Z_1 \Vert_{1-2\alpha} +  \Vert Z_2 \Vert_{1-2\alpha} ) \Vert Z_1 - Z_2  \Vert_{1-2\alpha} \\
&\quad + C \nu \Vert v_1 - v_2\Vert_{-\alpha}\\
&\leq C\lambda (\Vert v_1 \Vert_{-\alpha} +  \Vert v_2 \Vert_{-\alpha} ) \Vert v_1 - v_2  \Vert_{-\alpha} + C \lambda (\Vert Z_1 \Vert_{-\alpha} +  \Vert Z_2 \Vert_{-\alpha} ) \Vert Z_1 - Z_2  \Vert_{-\alpha} \\
&\quad + C \nu \Vert v_1 - v_2\Vert_{-\alpha}.
\end{align*}

Here, we used inequality~\eqref{eq:estb} to go from the second to the third line, the Poincar\'e inequality~\eqref{eq:poinca}, the fact that, for $\alpha \geq \gamma$, we have $\Vert f \Vert_\alpha \geq \Vert f\Vert_\gamma$, and the fact (proved in Lemma~\ref{lem:ellipticL}) that, for $z \in D((-\Delta)^{\alpha / 2})$,
$$
\Vert\Delta\LL^{-1}  z\Vert_{-\alpha} =\Vert\LL^{-1} \Delta z\Vert_{-\alpha} \leq C \Vert z \Vert_{-\alpha},
$$
since $1-2\alpha \leq -\alpha$ if $\alpha \geq 1$. 
Similarly, we have, for the terms in $N_2$, 
\begin{equation*}
\begin{aligned}
& \Vert N_2(v_1, Z_1) - N_2(v_2, Z_2)\Vert_{-\alpha} \\
&\leq \Vert \boldb(B_1 - B_2, u_1) + \boldb(B_2 , u_1 -u_2)  \Vert_{-\alpha} + \Vert \boldb(u_2 - u_1, B_2) + \boldb(u_1 , B_2 -B_1)  \Vert_{-\alpha} \\
&\leq C\|B_1 -B_2\|_0 \Vert B_1 - B_2\Vert_1 \Vert u_1\Vert_1 + C \|B_2\|_0 \Vert B_2\Vert_1 \Vert u_1- u_2 \Vert_1 +\\
&\quad  +C\|u_1 -u_2\|_0 \Vert u_1 - u_2\Vert_1 \Vert B_2 \Vert_1 + C \|u_1\|_0 \Vert u_1\Vert_1 \Vert B_1- B_2 \Vert_1 \\
&\leq C\lambda (\Vert v_1 \Vert_{-\alpha} +  \Vert v_2 \Vert_{-\alpha} ) \Vert v_1 - v_2  \Vert_{-\alpha} + C \lambda (\Vert Z_1 \Vert_{-\alpha} +  \Vert Z_2 \Vert_{-\alpha} ) \Vert Z_1 - Z_2  \Vert_{-\alpha}.
\end{aligned}
\end{equation*}
We conclude that the mapping $(N_1, N_2)$ is locally Lipschitz in $D((-\Delta)^{-\alpha/2}))$, which concludes the existence proof by an application of the Picard--Lindel\"of theorem. Finally, the continuation criterion is evident from the fact that 
$$
\Vert \LL^{-1} u \Vert_{-\alpha} \geq C \Vert u \Vert_{\alpha}.
$$
This concludes the proof.
\end{proof}

We show that high norms are propagated by the system~\eqref{eq:MHDV1}--\eqref{eq:MHDV3}.

\begin{theorem}[Global existence of solutions to~\eqref{eq:MHDV1}--\eqref{eq:MHDV3}]\label{thm:global}
Let $(u_0, B_0)$ in $D((-\Delta)^{\alpha/2}))$, with $\alpha \geq 1$, satisfying~\eqref{eq:initcond} and~\eqref{eq:meanz}. Then, there exist $(u, B) \in L^\infty(0,\infty; D((-\Delta)^{\alpha/2})))$ which solve~\eqref{eq:MHDV1}--\eqref{eq:MHDV3} in the sense of distributions, and such that
\begin{equation*}
(u, B)|_{t= 0} = (u_0, B_0).
\end{equation*}
\end{theorem}

\begin{proof}[Proof of Theorem~\ref{thm:global}]

Due to the continuation criterion of Proposition~\ref{prop:localMHD}, we only need to show that $\Vert u\Vert_{\alpha}+\Vert B\Vert_{\alpha}$ is bounded a-priori in terms of initial data. We have, taking the $L^2 $ inner product of the momentum equation with $u$,
\begin{equation*}
\begin{aligned}
&\frac 12 \frac{\text{d}}{\text{d}t}(\Vert u \Vert^2_{\alpha})=  ( u,  \p_t \LL u) =  - \nu \Vert u \Vert_1^2 + (\boldb(B,B),u) - (\boldb(u,u),u)  =  - \nu \Vert u \Vert_0^2 - (\boldb(B,u),B). 
\end{aligned}
\end{equation*}
On the other hand, we have, taking the $L^2$ inner product of the induction equation with $B$,
\begin{equation*}
\frac 12 \frac{\text{d}}{\text{d}t}(\Vert B \Vert^2_{\alpha})=  ( B,  \p_t \LL B) =  -(\boldb(u,B),B) + (\boldb(B,u),B).
\end{equation*}
Recalling that $(\boldb(u,B),B)= 0$, and summing the two previous equations, we finally get, for all times $t_2 \geq t_1\geq 0$,
\begin{equation}\label{eq:energybalance}
\begin{aligned}
&\Vert u(\cdot, t_2) \Vert^2_{\alpha}+ \Vert B(\cdot, t_2) \Vert^2_{\alpha} + \nu \int_{t_1}^{t_2} \Vert \nabla u(\cdot, s) \Vert^2 \de s \leq \Vert u(\cdot, t_1) \Vert^2_{\alpha}+ \Vert B(\cdot, t_1) \Vert^2_{\alpha},
\end{aligned}
\end{equation}
which provides the required a-priori control. These a-priori estimates are only formal, but can be made rigorous by Galerkin approximation.
\end{proof}

We show that the system~\eqref{eq:MHDV1}--\eqref{eq:MHDV3} propagates higher regularity.

\begin{proposition}[Higher regularity of solutions to~\eqref{eq:MHDV1}--\eqref{eq:MHDV3}]\label{thm:reg}
Let $(u_0, B_0)$ in $D((-\Delta)^{\beta/2}))$, with $\beta \geq \alpha\geq 1$. Then, the solution $(u, B)$ to the system~\eqref{eq:MHDV1}--\eqref{eq:MHDV3} constructed in Theorem~\eqref{thm:global}  satisfies the stronger bounds:
\begin{equation*}
\Vert (u,B) \Vert_{L^\infty(0,T; D((-\Delta)^{\beta/2})))} \leq C(\Vert u_0\Vert_{\beta},  \Vert B_0\Vert_{\beta}, T).
\end{equation*}
\end{proposition}

\begin{proof}[Sketch of proof] We only provide a sketch of the proof in the case $\beta = \alpha + k$, with $k \in \N$. Let us deal with the case $k = 1$. We formally\footnote{These formal estimates can be made rigorous by Galerkin approximation.} differentiate equations~\eqref{eq:MHDV1}--\eqref{eq:MHDV3} by $\partial_i$, and obtain
\begin{align*}
&\p_t (-\Delta)^\alpha \p_i u - \Delta(\p_i u) = \mathbb{P}(-\p_i u \cdot \nabla u-u \cdot \nabla \p_i u + \p_i B \cdot \nabla B +B \cdot \nabla \p_i B ),\\
&\p_t (-\Delta)^\alpha \p_i B =\mathbb{P}( -\p_i u \cdot \nabla B- u \cdot \nabla  \p_i B + \p_i B \cdot \nabla u + B \cdot \nabla \p_i u).
\end{align*}

Multiplying the first equation by $\p_i u$ and the second equation by $\p_i B$, summing over $i$, and integrating we get, using the fact that $(\boldb(a,b),c) = - (\boldb(a,c),b)$,
\begin{equation*}
\begin{aligned}
&\p_t ( \Vert u \Vert^2_{\alpha +1} +   \Vert B \Vert^2_{\alpha +1} ) + \nu \Vert u \Vert^2_2\\
&= -(\boldb(\p_i u, u), \p_i u)+(\boldb(\p_i B, B), \p_i u),\\
& \quad  - (\boldb(\p_i u, B), \p_i B) + (\boldb(\p_i B, u), \p_i B).
\end{aligned}
\end{equation*}
Now, by Sobolev embedding, it is clear that
$$
|(\boldb(\p_i u, u), \p_i u)| \leq C \Vert u\Vert_2^{\frac 32} \Vert u\Vert_1^{\frac 32}, \qquad |(\boldb(\p_i B, B), \p_i u)| \leq C \Vert u\Vert_2^{\frac 12} \Vert u\Vert_1^{\frac 12}\Vert B\Vert_2 \Vert B\Vert_1.
$$
Using the a-priori energy estimate, in conjunction with Gr\"onwall's inequality, we obtain
\begin{equation}
(u,B) \in L^\infty(0,T; D((-\Delta)^{(\alpha+1)/2})),
\end{equation}
i.e.~the required bound.

The proof for larger $k$ is similar, and we omit it here.
\end{proof}

\subsection{Magnetic potential}

We state and prove a lemma on the existence of the magnetic potential in the periodic case.

\begin{lemma}[Existence of the magnetic potential]\label{lem:magneticpot}
Let $(u, B)$ such that $$(u, B) \in L^\infty(0,T; D((-\Delta)^{\alpha/2}))$$ solving the system~\eqref{eq:MHDV1}--\eqref{eq:MHDV3}, according to Theorem~\ref{thm:global}, with divergence-free initial data $(u_0, B_0) \in D((-\Delta)^{\alpha/2})$. Let us consider the following initial value problem, for an unknown vector field $\Psi$:
\begin{equation}\label{eq:topotential}
\begin{aligned}
&\p_t \LL \Psi = u \times B,\\
&\Psi|_{t=0} = \Psi_0.
\end{aligned}
\end{equation}
Here, $\Psi_0$ is the unique $\dot H^{\alpha+1}(\T^3)$ vector field satisfying the following two properties:
\begin{equation}\label{eq:initcondpsi}
\nabla \times \Psi_0 = B_0, \qquad \text{div } \Psi_0 = 0.
\end{equation}
Under these conditions, we have that the system~\eqref{eq:topotential} has a global solution $\Psi \in L^\infty(0,\infty; D((-\Delta)^{(\alpha+1)/2}))$ which satisfies~\eqref{eq:topotential} and~\eqref{eq:initcondpsi}. Furthermore, the following equality holds true for all $t \geq 0$ and $x \in \T^3$.
\begin{equation}\label{eq:potential}
\nabla \times \Psi = B.
\end{equation} 
\end{lemma}

\begin{proof}[Proof of Lemma~\ref{lem:magneticpot}]
The existence and regularity parts are standard. To prove relation \eqref{eq:potential}, we take the curl of the evolution equation in~\eqref{eq:topotential}, in order to obtain
\begin{equation}
\p_t \LL (\nabla \times \Psi) = \nabla \times (u \times B) = \p_t \LL B.
\end{equation}
Integrating in time and using the initial conditions, this gives $\LL(\nabla \times \Psi) =  \LL B$ at all times $t \geq 0$. We conclude by the fact that the kernel of $\LL$ in $D((-\Delta)^{\alpha/2})$ is trivial.
\end{proof}

\section{Proof of Proposition~\ref{prop:example1}}\label{app:growth}

In this section, we provide a particular solution to the $3D$ MHD equation which, in the infinite time limit, creates discontinuities of $B$. The example is classical, and it is essentially obtained by decoupling the momentum and the induction equation, imposing that the $B$ field is always vertical.

This example should be contrasted with the result obtained in Section~\ref{sec:relax}, and it shows that, in general, for the full MHD system may not relax to a regular MHS equilibrium in the infinite time limit. This is different from the situation in the Voigt--MHD system.

\begin{proof}[Proof of Proposition~\ref{prop:example1}]
Let us consider the three dimensional torus obtained from the box $[-1,1] \times [-1,1] \times [-1,1]$ with opposite sides identified.

Suppose that $u$ has the following form: $u = (u_1(x,y), u_2(x,y),0)$, and that $B$ has the following form: $B= (0,0,B_3(t,x,y))$.

Let us consider the following stream function for $u$:
$$
\Psi := \sin(2\pi x) \sin (2 \pi y),
$$
and we let $u_1 := \partial_y \Psi$, $u_2 = -\partial_x\Psi$.  Under these conditions, $u$ satisfies the steady Euler equations:
$$
u \cdot \nabla u + \nabla p =0, \qquad \dive u = 0.
$$
Let us moreover evolve $B_3$ according to the following transport equation:
\begin{equation}\label{eq:beqn}
\partial_t B_3(t,x,y) + u \cdot \nabla B_3(t,x,y) = 0. 
\end{equation}
With these choices, we have that the pair $(u,B)$ satisfies the ideal 3D MHD equations:
\begin{equation}
\begin{aligned}
&\p_t u + u \cdot \nabla u+ \nabla p =   B \cdot \nabla B,\\
&\p_t B + u \cdot \nabla  B - B \cdot \nabla u=0,\\
&\dive u = 0, \qquad \dive B = 0.
\end{aligned}
\end{equation}
Since $u$ is constant in time, we only need to specify initial data for $B$. We set $B_0 (x,y) = \chi(y)$, where $\chi$ is a smooth and periodic function, $\chi: [-1,1] \to \R$, with the property that $\chi(y) = y$ for $|y| \leq 1/2$.
We show that, locally around the origin, the gradient of $B_3$ grows exponentially in time.
 
Restricting equation~\eqref{eq:beqn} to the $y$-axis, we have
\begin{equation}\label{eq:transportony}
\partial_t B_3 - \sin (y) \partial_y B_3 = 0.
\end{equation}
Let us define a function $y(t,a)$ by the ODE ($a$ is the Lagrangian label):
$y'(t,a) = - \sin(y(t,a))$, with $y(0) = a$. 
Then, it can be easily checked that, for all positive $t$ and for all $a$ such that $|a| \leq 1/2$ 
\begin{equation}\label{eq:todiff}
B_3(t, 0, y(t,a), z) = B_0(0,a,z).
\end{equation} Note that have the following relation satisfied by $y(t,a)$:
$$
\tan\Big({\frac{y(t,a)}{2}}\Big) = \tan\Big(\frac a 2\Big) e^{-t},
$$
Differentiating relation~\eqref{eq:todiff} and calculating the result at $a = 0$, we have that
$$
|\p_y B_3 (t,0,0,z)| = C'e^t,
$$
for a positive constant $C'$, thereby proving the claim.
\end{proof}

\subsection*{Acknowledgements}

We thank Theo Drivas and Huy Nguyen for several insightful comments. P.C.'s research was partially supported by
NSF grant DMS-2106528 and by Simons Foundation grant 601960.

\bibliography{voigt.bib} 

\begin{thebibliography}{10}

\bibitem{alfven}
H.~Alfv{\'e}n.
\newblock Existence of electromagnetic-hydrodynamic waves.
\newblock {\em Nature}, 150(3805):405--406, 1942.

\bibitem{perspectives}
G.~K. Batchelor, H.~K. Moffatt, and M.~G. Worster, editors.
\newblock {\em Perspectives in fluid dynamics}.
\newblock Cambridge University Press, Cambridge, 2000.
\newblock A collective introduction to current research.

\bibitem{beekie2021}
Rajendra Beekie, Susan Friedlander, and Vlad Vicol.
\newblock On {M}offatt's magnetic relaxation equations.
\newblock {\em Comm. Math. Phys.}, 390(3):1311--1339, 2022.

\bibitem{biskamp}
Dieter Biskamp.
\newblock {\em Nonlinear magnetohydrodynamics}, volume~1 of {\em Cambridge
  Monographs on Plasma Physics}.
\newblock Cambridge University Press, Cambridge, 1993.

\bibitem{boyerfabrie}
Franck Boyer and Pierre Fabrie.
\newblock {\em Mathematical tools for the study of the incompressible
  {N}avier-{S}tokes equations and related models}, volume 183 of {\em Applied
  Mathematical Sciences}.
\newblock Springer, New York, 2013.

\bibitem{brenier1}
Y.~Brenier.
\newblock Topology-preserving diffusion of divergence-free vector fields and
  magnetic relaxation.
\newblock {\em Comm. Math. Phys.}, 330(2):757--770, 2014.

\bibitem{brunolaurence}
O.~P. Bruno and P.~Laurence.
\newblock Existence of three-dimensional toroidal {MHD} equilibria with
  nonconstant pressure.
\newblock {\em Comm. Pure Appl. Math.}, 49(7):717--764, 1996.

\bibitem{caoetal2004}
Yanping Cao, Evelyn~M. Lunasin, and Edriss~S. Titi.
\newblock Global well-posedness of the three-dimensional viscous and inviscid
  simplified {B}ardina turbulence models.
\newblock {\em Commun. Math. Sci.}, 4(4):823--848, 2006.

\bibitem{constbook}
P.~{Constantin} and C.~{Foias}.
\newblock {\em {Navier-Stokes equations.}}
\newblock Chicago, IL etc.: University of Chicago Press, 1988.

\bibitem{constlavicol}
P.~Constantin, J.~La, and V.~Vicol.
\newblock Remarks on a paper by {G}avrilov: {G}rad-{S}hafranov equations,
  steady solutions of the three dimensional incompressible {E}uler equations
  with compactly supported velocities, and applications.
\newblock {\em Geom. Funct. Anal.}, 29(6):1773--1793, 2019.

\bibitem{cdg2020}
Peter Constantin, Theodore~D. Drivas, and Daniel Ginsberg.
\newblock Flexibility and rigidity in steady fluid motion.
\newblock {\em Comm. Math. Phys.}, 385(1):521--563, 2021.

\bibitem{cdg2021}
Peter Constantin, Theodore~D. Drivas, and Daniel Ginsberg.
\newblock On quasisymmetric plasma equilibria sustained by small force.
\newblock {\em Journal of Plasma Physics}, 87(1):905870111, 2021.

\bibitem{elgindimasmoudi}
Tarek~M. Elgindi and Nader Masmoudi.
\newblock {$L^\infty$} ill-posedness for a class of equations arising in
  hydrodynamics.
\newblock {\em Arch. Ration. Mech. Anal.}, 235(3):1979--2025, 2020.

\bibitem{enciso2021}
A.~Enciso, A.~Luque, and D.~Peralta-Salas.
\newblock {MHD equilibria with nonconstant pressure in nondegenerate toroidal
  domains}.
\newblock Preprint: \url{https://arxiv.org/abs/2104.08149}, 2021.

\bibitem{gavrilov}
A.~V. Gavrilov.
\newblock A steady {E}uler flow with compact support.
\newblock {\em Geom. Funct. Anal.}, 29(1):190--197, 2019.

\bibitem{hill}
M.~J.~M. Hill.
\newblock On a spherical vortex.
\newblock {\em Philosophical Transactions of the Royal Society of London. A},
  185:213--245, 1894.

\bibitem{kiselev}
A.~Kiselev and V.~\v{S}ver\'{a}k.
\newblock Small scale creation for solutions of the incompressible
  two-dimensional {E}uler equation.
\newblock {\em Ann. of Math. (2)}, 180(3):1205--1220, 2014.

\bibitem{kraushudson}
B.~F. Kraus and S.~R. Hudson.
\newblock Theory and discretization of ideal magnetohydrodynamic equilibria
  with fractal pressure profiles.
\newblock {\em Physics of Plasmas}, 24(9):092519, 2017.

\bibitem{kruskalkulsrud}
M.~D. Kruskal and R.~M. Kulsrud.
\newblock Equilibrium of a magnetically confined plasma in a toroid.
\newblock {\em Phys. Fluids}, 1:265--274, 1958.

\bibitem{lariostiti}
{A.} Larios and {E.} Titi.
\newblock On the higher-order global regularity of the inviscid
  voigt-regularization of three-dimensional hydrodynamic models.
\newblock {\em Discrete and Continuous Dynamical Systems B}, 14(1531-3492),
  2010.

\bibitem{lariostiti2014}
Adam Larios and Edriss~S. Titi.
\newblock Higher-order global regularity of an inviscid {V}oigt-regularization
  of the three-dimensional inviscid resistive magnetohydrodynamic equations.
\newblock {\em J. Math. Fluid Mech.}, 16(1):59--76, 2014.

\bibitem{levantramostiti}
Boris Levant, Fabio Ramos, and Edriss~S. Titi.
\newblock On the statistical properties of the 3{D} incompressible
  {N}avier-{S}tokes-{V}oigt model.
\newblock {\em Commun. Math. Sci.}, 8(1):277--293, 2010.

\bibitem{linshitztiti2007}
Jasmine~S. Linshiz and Edriss~S. Titi.
\newblock Analytical study of certain magnetohydrodynamic-{$\alpha$} models.
\newblock {\em J. Math. Phys.}, 48(6):065504, 28, 2007.

\bibitem{moffatt1}
H.~K. Moffatt.
\newblock Magnetostatic equilibria and analogous {E}uler flows of arbitrarily
  complex topology. {I}. {F}undamentals.
\newblock {\em J. Fluid Mech.}, 159:359--378, 1985.

\bibitem{moffatt2}
H.~K. Moffatt.
\newblock Magnetostatic equilibria and analogous {E}uler flows of arbitrarily
  complex topology. {P}art 2. {S}tability considerations.
\newblock {\em Journal of Fluid Mechanics}, 166:359–378, 1986.

\bibitem{oskolkov73}
A.~P. Oskolkov.
\newblock The uniqueness and solvability in the large of boundary value
  problems for the equations of motion of aqueous solutions of polymers.
\newblock {\em Zap. Nau\v{c}n. Sem. Leningrad. Otdel. Mat. Inst. Steklov.
  (LOMI)}, 38:98--136, 1973.
\newblock Boundary value problems of mathematical physics and related questions
  in the theory of functions, 7.

\bibitem{pasqualotto2020}
F.~Pasqualotto.
\newblock {\em {{N}onlinear {W}aves in {G}eneral {R}elativity and {F}luid
  {D}ynamics}}.
\newblock PhD thesis, Princeton University, 2020.

\bibitem{ramostiti}
Fabio Ramos and Edriss~S. Titi.
\newblock Invariant measures for the 3{D} {N}avier-{S}tokes-{V}oigt equations
  and their {N}avier-{S}tokes limit.
\newblock {\em Discrete Contin. Dyn. Syst.}, 28(1):375--403, 2010.

\bibitem{taylor1986}
J.~B. Taylor.
\newblock Relaxation and magnetic reconnection in plasmas.
\newblock {\em Rev. Mod. Phys.}, 58:741--763, Jul 1986.

\bibitem{vainshtein}
S.~I. Va\u{\i}nshte\u{\i}n.
\newblock Force-free magnetic fields with constant alpha.
\newblock In {\em Topological aspects of the dynamics of fluids and plasmas
  ({S}anta {B}arbara, {CA}, 1991)}, volume 218 of {\em NATO Adv. Sci. Inst.
  Ser. E: Appl. Sci.}, pages 177--193. Kluwer Acad. Publ., Dordrecht, 1992.

\bibitem{Yudovich74}
V.I. Yudovich.
\newblock On the loss of smoothness of the solutions of {E}uler's equations
  with time.
\newblock {\em Dinamika Sploshnoi Sredy (Dynamics of Continuous Media)
  (Novosibirsk, Russia)}, 16:71--78, 1974.

\end{thebibliography}
\bibliographystyle{plain}

\end{document}